\documentclass[11pt,a4paper]{article}
\usepackage[margin=1.1cm]{geometry}
\usepackage{amsmath,amsthm,amsfonts,amssymb,amscd,cite,graphicx}
\usepackage{latexsym}
\usepackage{enumitem}

\usepackage[usenames, dvipsnames]{color}
\definecolor{mygray}{gray}{0.6}

\usepackage{titlesec}
\titleformat{\section}
{\normalfont\fontsize{12}{15}\bfseries}{\thesection}{1em.}{}

\newtheorem{proposition}{Proposition}[section]

\newtheorem{corollary}{Corollary}[section]

\newtheorem{theorem}{Theorem}[section]

\allowdisplaybreaks[4]

\let\oldbibliography\thebibliography
\renewcommand{\thebibliography}[1]{%
  \oldbibliography{#1}%
  \setlength{\itemsep}{-2pt}%
}

\begin{document}

\baselineskip=0.20in

\makebox[\textwidth]{%
\hglue-15pt
\begin{minipage}{0.6cm}	
\vskip9pt
\end{minipage} \vspace{-\parskip}
\hfill
\begin{minipage}[t]{6.5cm}
\normalsize {\it }  {\bf}
\end{minipage}}
\vskip36pt

\noindent
{\large \bf Distribution of the inversion statistic on run-sorted permutations}\\

\noindent Toufik Mansour\\

\noindent
\footnotesize{\it Department of Mathematics, University of Haifa, 3498838 Haifa, Israel\\
Email: tmansour@univ.haifa.ac.il}\\

\normalsize\noindent
\noindent Olivia Nabawanda\\

\noindent
\footnotesize {\it Department of Mathematics, Makerere University, Kampala, Uganda\\
Email: nabawandaolivia100@gmail.com}\\

\normalsize\noindent
Mark Shattuck\\

\noindent
\footnotesize {\it Department of Mathematics, University of Tennessee,
37996 Knoxville, TN\\
Email: mark.shattuck2@gmail.com}\\

\setcounter{page}{1} \thispagestyle{empty}

\baselineskip=0.20in

\normalsize

\begin{abstract}
Let $\pi=\pi_1\cdots \pi_n$ be a permutation.  We say that $\pi$ is \emph{run-sorted} if $\pi_1=1$ and the entries immediately following the descent positions of $\pi$ form an increasing sequence.  Let $\mathcal{R}_n$ denote the set of run-sorted permutations of length $n$, which has cardinality given by the Bell number $B_{n-1}$ for all $n \geq 1$.  In this paper, we consider the joint distribution $A_n(q,u)$ on $\mathcal{R}_n$ for the parameters tracking the numbers of inversions and runs leading to a new polynomial generalization of the Bell numbers.  Among our results, we find a general recurrence for $A_n(q,u)$, from which one may derive explicit formulas for the total numbers of inversions or runs in all the members of $\mathcal{R}_n$ as well as for the sign-balance on $\mathcal{R}_n$ of either parameter.  A simple expression for the Eulerian generating function for $A_n(q,1)$ may be found upon making use of Gessel's $q$-exponential formula which can be extended to general $u$.  Finally, a formula is found by a direct argument for the maximum number of inversions within a member of $\mathcal{R}_n$. \medskip

\end{abstract}

 \medskip


\section{Introduction}

Let $\mathcal{S}_n$ denote the set of permutations of $[n]=\{1,\ldots,n\}$.  Given $\pi=\pi_1\cdots \pi_n \in \mathcal{S}_n$, expressed in the one-line notation, an \emph{inversion} of $\pi$ refers to an ordered pair $(i,j)$ such that $1 \leq i < j \leq n$ and $\pi_i>\pi_j$.  The total number of inversions of $\pi$ is denoted by $\text{inv}(\pi)$.   Let $\text{run}(\pi)$ denote the number of runs of $\pi$, where a \emph{run} refers to a maximal increasing sequence of consecutive entries.  For example, if $\pi=256143 \in \mathcal{S}_6$, then we have $\text{inv}(\pi)=8$, the inversions of $\pi$ corresponding to the following pairs of entries:
$$\{(2,1),\,(4,3),\,(5,1),\,(5,3),\,(5,4),\,(6,1),\,(6,3),(6,4)\}.$$
The runs of $\pi$ are given by 256, 14 and 3, and hence $\text{run}(\pi)=3$.

The number of inversions in a permutation is a measure of the extent to which a permutation deviates from the natural order.  The fact that the distribution of inv on $\mathcal{S}_n$ is the $q$-analogue of $n!$ dates from at least an 1898 paper of Muir \cite{Muir} and it was shown thereafter that inv has the same distribution on $\mathcal{S}_n$ as the major index (see MacMahon \cite{Mac}).  Different aspects of the inv distribution on $S_n$ have been subsequently studied including the asymptotic normality (see, for example, \cite{Mar} and references contained therein).  For comparable work concerning the parameters on $\mathcal{S}_n$ tracking the number of runs or alternating runs, see, e.g., \cite{SMa1,SMa2,Zh}.  Recently, an analogue of inversions on colored permutations \cite{ARV} was introduced, along with its restriction to derangements and involutions.  Here, we consider a restriction of inv to a certain subset of $\mathcal{S}_n$ related to set partitions.

A \emph{partition} of a finite $S$  is a collection of nonempty, pairwise disjoint subsets, called \emph{blocks}, whose union is $S$.  Let $\mathcal{P}_n$ denote the set of partitions of $[n]$ and recall $\mathcal{P}_n$ has cardinality given by the $n$-th Bell number $B_n$ for all $n\geq 0$, see A000110 in the OEIS \cite{Sl}.  Given $\pi \in \mathcal{P}_n$, the \emph{flattening} of $\pi$, denoted by $\text{flat}(\pi)$, is obtained by writing the elements of each block in increasing order, arranging the blocks from left to right in ascending order of their first (= smallest) elements and then considering the permutation that results by removing the parentheses enclosing each block. We refer to a permutation obtained in this manner as a \emph{flattened partition}.  The study of flattened partitions was initiated by Callan \cite{Callan}, who considered the avoidance of classical patterns of length three.  Many further results concerning avoidance by flattened partitions or related statistics were obtained soon thereafter and the notion of flattening has recently been extended to other discrete structures, such as Catalan words \cite{BHR}, Stirling permutations \cite{BEF} and parking functions \cite{EHM}.

Note that the flattening operation is not injective, and thus it may be interesting to consider the set of distinct permutations of $[n]$ that arise from $\text{flat}(\pi)$ as $\pi$ ranges over all members of $\mathcal{P}_n$.  Let $\mathcal{R}_n$ denote this set of permutations of $[n]$.  For example, if $n=4$, then $\mathcal{R}_4=\{1234, 1243, 1324, 1342, 1423\}$.  Note that members of $\mathcal{R}_n$ are characterized by the property that the elements starting runs form an increasing subsequence, which follows from the ordering applied to the blocks of a partition prior to being flattened.   In particular, all members of $\mathcal{R}_n$ must start with 1. For example, if $n=8$, then the permutation $13827456$ is a member of $\mathcal{R}_8$, as it has runs 138, 27, 456 whose first entries 1, 2, 4 are in increasing order, whereas $13845627$ does not belong to $\mathcal{R}_8$, since the first entries 1, 4, 2 are not increasing.  Members of $\mathcal{R}_n$ are referred to as \emph{run-sorted} permutations, a terminology introduced in \cite{ANab} that we adopt throughout this paper.

It is known that $\mathcal{R}_n$ has cardinality given by $B_{n-1}$ for all $n \geq 1$. Nabawanda et al.\ \cite{NRB} studied various aspects of the run distribution on $\mathcal{R}_n$ and defined a bijection between $\mathcal{R}_n$ and $\mathcal{P}_{n-1}$.  In \cite{NR}, the problem of avoidance of one or two classical patterns of length three by members of $\mathcal{R}_n$ was addressed leading to connections to several well-known sequences.  In \cite{MSh}, the distributions on $\mathcal{R}_n$ tracking the number of occurrences of consecutive patterns of length three were studied and the corresponding generating function was found in each case.  Here, we consider the joint distribution polynomial $A_n(q,u)$ defined by
\begin{equation}\label{Anqdef}
A_n(q,u)=\sum_{\pi \in \mathcal{R}_n}q^{\text{inv}(\pi)}u^{\text{run}(\pi)}, \qquad n \geq 1,
\end{equation}
where $q$ and $u$ are indeterminates.  Note that $A_{n}(1,1)=|\mathcal{R}_{n}|=B_{n-1}$ for all $n$.

The organization of this paper is as follows. In the next section, we establish several properties of the polynomial $A_n(q,u)$ starting with the defining recurrence, among them, the values of $A_n(q,u)$ when $q$ or $u$ equals $-1$ as well as the respective partial derivatives evaluated at $q=1$ or $u=1$.  These values may be expressed in terms of the Stirling, Bell and complementary Bell number sequences and we provide both explicit and asymptotic formulas.  Furthermore, a simple explicit formula in terms of the $q$-exponential function is found for the Eulerian generating function of the sequence $A_{n+1}(q,1)$ for $n\geq0$ wherein the $n!$ in the usual Maclaurin series is replaced by its $q$-analogue. In the third section, a formula is proven for the maximum number of inversions achieved by a member of $\mathcal{R}_n$, i.e., the degree of the polynomial $A_n(q,1)$. Our argument for this result entails showing that every member of $\mathcal{R}_n$ outside of those in a special subset can be modified in such a way so as to yield a member of $\mathcal{R}_n$ with a strictly greater number of inversions.  Finally, our argument may be extended to yield formulas for the maximum inv value achieved by members of $\mathcal{R}_n$ with a fixed number of runs.

A permutation $\pi=\pi_1\cdots \pi_n$ avoids the vincular pattern 32-1 if it contains no subsequence of the form $\pi_{i-1}\pi_i\pi_j$ for some $1<i< j\leq n$ such that $\pi_{i-1}>\pi_i>\pi_j$.  Let $\mathcal{V}_n$ denote the set of 32-1 avoiding permutations of $[n]$.  Note that a permutation avoids 32-1 if and only if the subsequence consisting of the first elements of its runs, excluding the first run, is increasing. Then the sets of permutations $\mathcal{R}_{n+1}$ and $\mathcal{V}_n$ are seen to be equivalent, upon deleting the element $1$ from each member of $\mathcal{R}_{n+1}$ and subtracting one from the other letters. Thus, the results here complement some recent work in \cite{Bev,QP} concerning the joint distribution of inversions and descents on $\mathcal{V}_n$.

Let us now recall some standard notation.  Given an indeterminate $q$, let $n_q=1+q\cdots+q^{n-1}$ for $n \geq 1$, with $0_q=0$.  Let $n_q^!=\prod_{i=1}^n i_q$ for $n \geq 1$, with $0_q^!=1$, denote the $q$-factorial.  The $q$-binomial coefficient $\binom{n}{k}_q$ for $n,k \geq0$ is defined by
\[
 \binom{n}{k}_q =
\begin{cases}
    \frac{n_q^!}{k_q^!(n-k)_q^!},   & \text{if } 0 \leq k \leq n; \\
    0,    & \text{otherwise}.
\end{cases}
\]
Given two sequences of real numbers $a_n$ and $b_n$, we write $a_n \sim b_n$ if $\lim_{n \rightarrow \infty}(a_n/b_n)=1$.  Finally, let $[x^m/m!]f(x)$ for a non-negative integer $m \geq0$ denote the coefficient of $x^m/m!$ in the exponential generating function $f(x)$.

\section{Joint distribution of runs and inversions}

The polynomials $A_n(q,u)$ defined by \eqref{Anqdef} satisfy the following recurrence relation.

\begin{theorem}\label{Anqrec}
If $n \geq 2$, then
\begin{equation}\label{Anqrece1}
A_n(q,u)=A_{n-1}(q,u)+u\sum_{k=1}^{n-2}q^k\binom{n-2}{k}_qA_{n-k-1}(q,u),
\end{equation}
with $A_1(q,u)=u$.
\end{theorem}
\begin{proof}
The initial condition when $n=1$ as well as the $n=2$ case of \eqref{Anqrece1} are clear.  So we argue \eqref{Anqrece1} for $n \geq 3$ and let $\pi \in \mathcal{R}_n$.  If $1$ and $2$ belong to the same run of $\pi$, then one gets $A_{n-1}(q,u)$ possibilities  in this case, upon deleting 1 and subtracting one from every other entry of $\pi$.

Now suppose 1 and 2 belong to different runs of $\pi$.  Since runs are ordered by increasing first elements, the initial run comprises 1 together with some subset $S$ of $[3,n]$, written in increasing order, with 2 starting the second run.  Let $|S|=k$, where $1 \leq k \leq n-2$.  Deleting the entire first run and standardizing the remaining entries, we obtain a member of $\mathcal{R}_{n-k-1}$.  This yields $uA_{n-k-1}(q,u)$, where the factor of $u$ accounts for the deleted first run.

Thus, it remains to determine the contribution towards the weight derived from the members of $S$.  First note that each member of $S$ occurs to the left of 2, which accounts for the factor of $q^k$.  Recall that $\binom{n}{m}_q$ gives the inv distribution for binary words consisting of $m$ 1's and $n-m$ 0's; see, e.g., \cite[p.\,26]{Stan}.  Upon encoding each member of $S$ by a 1, and each member of $[3,n]\backslash S$ by a 0, it follows that $\binom{n-2}{k}_q$ accounts for both the choice of $S$ and the inversions arising from elements of $S$ occurring to the left of those in $[3,n]\backslash S$ within $\pi$.  Combining the prior observations, we have a contribution of $uq^k\binom{n-2}{k}_qA_{n-k-1}(q,u)$ towards the distribution for each $k$.  Considering all possible $k$ yields the summation on the right-hand side of \eqref{Anqrece1} and completes the proof of \eqref{Anqrece1}.
\end{proof}

\noindent \emph{Remark:}  A recurrence equivalent to \eqref{Anqrece1} was found in \cite[Theorem 2.5]{QP} for the joint distribution on $\mathcal{V}_n$ of the number of inversions and descents via a lengthier argument considering Lehmer codes of 32-1 avoiding permutations.

\subsection{Explicit formulas in some particular cases}

Let $\mathcal{R}_{n,k}$ denote the subset of $\mathcal{R}_n$ consisting of those members that contain $k$ runs and let $A_{n,k}(q)$ be the restriction of $A_n(q,u)$ to $\mathcal{R}_{n,k}$ for $1 \leq k \leq \lfloor(n+1)/2 \rfloor$. Let $G_n(q)=\sum_{k=0}^n\binom{n}{k}_q$ for $n \geq 0$ denote the $n$-th \emph{Galois} number, which reduces to $2^n$ when $q=1$.  Recall that the $G_n(q)$ are given recursively by $G_{n+1}(q)=2G_{n}(q)+(q^n-1)G_{n-1}(q)$ for $n\geq 1$, see \cite{GR}.  There are the following explicit formulas for $A_{n,2}(q)$ and $A_{n,3}(q)$ in terms of $G_n(q)$.

\begin{proposition}\label{An,2(q)}
We have
\begin{equation}\label{An,2(q)e1}
A_{n,2}(q)=G_{n-1}(q)-n, \qquad n \geq3,
\end{equation}
and
\begin{equation}\label{An,2(q)e2}
A_{n,3}(q)=\sum_{m=1}^{n-4}\left(\binom{n-1}{m}_q-1\right)(G_{n-m-2}(q)-n+m+1),  \qquad n \geq 5.
\end{equation}
\end{proposition}
\begin{proof}
Let $a_n=A_{n,2}(q)$.  Comparing coefficients of $u^2$ on both sides of \eqref{Anqrece1}, we obtain the recurrence
\begin{equation}\label{An,2(q)e3}
a_n=a_{n-1}+\sum_{k=1}^{n-2}q^k\binom{n-2}{k}_q, \qquad n \geq 2,
\end{equation}
with $a_2=0$.  By induction using \eqref{An,2(q)e3} and the recurrence for $q$-binomial coefficients, one can show $a_n=\sum_{k=1}^{n-1}\binom{n-1}{k}_q-(n-1)$, which implies \eqref{An,2(q)e1}.  Alternatively, formula \eqref{An,2(q)e1} may be argued directly as follows.  Let $\pi \in \mathcal{R}_{n,2}$ and let $T$ denote the subset of $[2,n]$ comprising those elements greater than 1 that belong to the first run of $\pi$.  Note $\max(T)>\min([2,n]\backslash T)$ since $\pi$ contains two runs, and hence $T$ cannot be an interval of the form $[2,k]$ for some $k \geq 2$.  Upon encoding each element of $[2,n]$ by a 1 or 0 according to whether or not it belongs to $T$, we have that the weight of those $\pi$ for which $|T|=m$ is given by $\binom{n-1}{m}_q-1$, where the subtracted term accounts for the excluded subset $[2,m+1]$ and $1 \leq m \leq n-2$. Considering all possible $m$ then gives
$$a_n=\sum_{m=1}^{n-2}\left(\binom{n-1}{m}_q-1\right)=G_{n-1}(q)-n,$$
as desired.

For \eqref{An,2(q)e2}, let $\pi \in \mathcal{R}_{n,3}$, where $n \geq 5$.  Let $T$ be as before, and hence $T$ is not an interval.  If $|T|=m$, then $\pi$ containing three runs implies $m \in [n-4]$.  The choice of $T$, along with the inversions it renders, is accounted for by $\binom{n-1}{m}_q-1$ for each $m$, with the elements of $[2,n]\backslash T$ comprising a member of $\mathcal{R}_{n-m-1,2}$, when standardized.  By \eqref{An,2(q)e1}, these elements have weight $G_{n-m-2}(q)-(n-m-1)$, independent of the choice of $T$.  This implies that the weight of the members of $\mathcal{R}_{n,3}$ for which $|T|=m$ is given by
$$\left(\binom{n-1}{m}_q-1\right)(G_{n-m-2}(q)-n+m+1).$$
Considering all possible $m$ yields \eqref{An,2(q)e2} and completes the proof.
\end{proof}

\begin{proposition}\label{totruns}
The number of runs in all the members of $\mathcal{R}_n$ is given by $B_n-(n-1)B_{n-2}$ for $n \geq 2$.
\end{proposition}
\begin{proof}
Let $t_n=\sum_{\pi \in \mathcal{R}_n}\text{run}(\pi)$ for $n \geq 1$. Differentiating both sides of \eqref{Anqrece1} with respect to $u$, setting $u=1$ and noting $A_n(1,1)=B_{n-1}$, we obtain
$$t_n=t_{n-1}+\sum_{k=1}^{n-2}\binom{n-2}{k}B_{n-k-2}+\sum_{k=1}^{n-2}\binom{n-2}{k}t_{n-k-1}, \qquad n \geq 2.$$
By the recurrence for Bell numbers (see, e.g., \cite[p.\,34]{Stan}), we have equivalently
\begin{equation}\label{totrunse1}
t_n=B_{n-1}-B_{n-2}+\sum_{k=0}^{n-2}\binom{n-2}{k}t_{n-k-1}, \qquad n \geq 2,
\end{equation}
with $t_1=1$.

We proceed inductively, noting that the result is clear for $n=2$.  First observe the identity $B_{n+1}-B_n=\sum_{k=0}^{n-1}\binom{n-1}{k}B_{n-k}$ for $n\geq 1$, since the right side of the equality enumerates the partitions of $[n+1]$  where 1 and $n+1$ do not occur in the same block (for which there are $B_{n+1}-B_n$, by subtraction), upon considering the number $k$ of elements in the block containing 1.  Thus, by \eqref{totrunse1} and also the usual recurrence for Bell numbers, we have for $n \geq 2$,
\begin{align*}
t_{n+1}&=B_n-B_{n-1}+\sum_{k=0}^{n-1}\binom{n-1}{k}t_{n-k}=B_n-B_{n-1}+\sum_{k=0}^{n-1}\binom{n-1}{k}\left(B_{n-k}-(n-k-1)B_{n-k-2}\right)\\
&=B_n-B_{n-1}+B_{n+1}-B_n-\sum_{k=0}^{n-2}(n-k-1)\binom{n-1}{k}B_{n-k-2}\\
&=B_{n+1}-B_{n-1}-(n-1)\sum_{k=0}^{n-2}\binom{n-2}{k}B_{n-k-2}=B_{n+1}-nB_{n-1},
\end{align*}
which completes the induction and proof.
\end{proof}

The prior result may be argued combinatorially as follows.  Recall that the statistic on $\mathcal{R}_{n+1}$ tracking the number of runs is equivalent to the parameter on $\mathcal{P}_n$ recording one plus the number of non-singleton blocks (see \cite{NRB}, where a bijection is defined between $\mathcal{R}_{n+1}$ and $\mathcal{P}_n$ demonstrating this fact).  Thus, to establish $t_{n+1}=B_{n+1}-nB_{n-1}$, it suffices to show that there are $B_{n+1}-B_n-nB_{n-1}$ non-singleton blocks altogether in $\mathcal{P}_n$.  This may be achieved as follows.  First note that there are clearly $nB_{n-1}$ singleton blocks in $\mathcal{P}_n$, as each block $\{i\}$ for $i \in [n]$ occurs $B_{n-1}$ times.  Further, there are $B_{n+1}-B_n$ blocks altogether in $\mathcal{P}_n$.  To see this, note that marked members of $\mathcal{P}_n$ wherein one of the blocks is marked are in one-to-one correspondence with the subset of $\mathcal{P}_{n+1}$ consisting of those partitions in which the singleton $\{n+1\}$ does not occur, upon adding the element $n+1$ to the marked block within a member of the first set.  Since there are $B_{n+1}-B_n$ members of $\mathcal{P}_{n+1}$ in which $\{n+1\}$ does not occur, by subtraction, we get the same number of blocks altogether in $\mathcal{P}_n$.  Thus, the desired formula stated above for the number of non-singleton blocks in $\mathcal{P}_n$ follows by subtraction.

Let $T_{n,k}$ denote the \emph{associated} Stirling number of the second kind (see \cite[A008299]{Sl}), which enumerates the partitions of $[n]$ into $k$ non-singleton blocks and is given recursively by $T_{n+1,k}=kT_{n,k}+nT_{n-1,k-1}$.  We have the following explicit formula for the cardinality of $\mathcal{R}_{n+1,k+1}$ in terms of $T_{n,k}$.

\begin{proposition}\label{R_{n,k}expl}
If $n \geq 0$ and $0 \leq k \leq \lfloor n/2 \rfloor$, then
\begin{equation}\label{R_{n,k}exple1}
|\mathcal{R}_{n+1,k+1}|=\sum_{i=0}^{n-2k}\binom{n}{i}T_{n-i,k}.
\end{equation}
Thus, we have for $n \geq 1$ the identity
\begin{equation}\label{R_{n,k}exple2}
B_{n+1}-nB_{n-1}=\sum_{k=0}^{\lfloor n/2 \rfloor}\sum_{i=0}^{n-2k}(k+1)\binom{n}{i}T_{n-i,k}.
\end{equation}
\end{proposition}
\begin{proof}
By \cite[Theorem 19]{NRB}, to determine $|\mathcal{R}_{n+1,k+1}|$, we equivalently count members of $\mathcal{P}_n$ that contain exactly $k$ non-singleton blocks. Let $i$ denote the number of singletons within such a member of $\mathcal{P}_n$.  Then $0 \leq i \leq n-2k$ since each of the other blocks is of size at least two and there are $\binom{n}{i}T_{n-i,k}$ possibilities for each $i$. Considering all possible $i$ yields \eqref{R_{n,k}exple1}.  Identity \eqref{R_{n,k}exple2} then follows from \eqref{R_{n,k}exple1} and Proposition \ref{totruns}, together with the fact $t_{n+1}=\sum_{k=0}^{\lfloor n/2 \rfloor}(k+1)|\mathcal{R}_{n+1,k+1}|$.
\end{proof}

Using \cite[Theorem 19]{NRB} and the exponential formula (see, e.g., \cite[Theorem 6.1.4]{Wag}), one can show
\begin{equation}\label{Txuform}
\sum_{n\geq0}A_{n+1}(1,u)\frac{x^n}{n!}=ue^{u(e^x-1)+x(1-u)}.
\end{equation}

From \eqref{Txuform}, we may derive an explicit formula for the sign-balance of the run parameter on $\mathcal{R}_n$ as follows. Let $\mathcal{R}_n^{(e)}$ and $\mathcal{R}_n^{(o)}$ denote the subsets of $\mathcal{R}_n$ whose members contain an even or an odd number or runs, respectively.  Let $B_n^*$ denote the $n$-th complementary Bell number defined recursively by $B_{n+1}^*=-\sum_{i=0}^n\binom{n}{i}B_i^*$ for $n \geq0$, with $B_0^*=1$; see, e.g., \cite{RUC} or \cite[A000587]{Sl}. Note that the $u=-1$ case of \eqref{Txuform} may be written as
$$\sum_{n\geq0}A_{n+1}(1,-1)\frac{x^n}{n!}=-e^{1+2x-e^x}=\frac{d}{dx}\left(e^{1-e^x}\right)-\frac{d^2}{dx^2}\left(e^{1-e^x}\right).$$
Recalling the fact $e^{1-e^x}=\sum_{n\geq0}B_n^*\frac{x^n}{n!}$ yields the following result.

\begin{proposition}\label{sgnbal}
If $n \geq 1$, then $A_{n}(1,-1)=|\mathcal{R}_{n}^{(e)}|-|\mathcal{R}_{n}^{(o)}|=B_{n}^*-B_{n+1}^*$.
\end{proposition}

\noindent \emph{Remark:} Extracting the coefficient of $x^n$ in \eqref{Txuform} when $u=-1$ in two different ways yields the additional formulas
$$A_{n+1}(1,-1)=-\sum_{i=0}^n\binom{n}{i}2^iB_{n-i}^*=\sum_{i=0}^n\binom{n}{i}B_{n-i+1}^*,$$
which give identities for $B_n^*$ when taken together with the result of Corollary \ref{sgnbal}.  For an asymptotic estimate of $B_n^*$, and hence of $A_{n}(1,-1)$, see \cite[Theorem 1]{Ya}.    \medskip

Let $S_{n,k}$ denote the classical Stirling number of the second kind, see \cite[A008277]{Sl}.  Our final result in this subsection expresses the cardinality of $\mathcal{R}_{n+1,k+1}$ in terms of Stirling numbers.

\begin{theorem}\label{Rnkcard}
If $n \geq1$ and $0 \leq k \leq \lfloor n/2 \rfloor$, then
\begin{equation}\label{Rnkcarde1}
|\mathcal{R}_{n+1,k+1}|=\sum_{i=0}^k\sum_{j=i}^n (-1)^{k-i}\binom{n}{j}\binom{n-j}{k-i}S_{j,i}.
\end{equation}
\end{theorem}
\begin{proof}
We provide two proofs of \eqref{Rnkcarde1}, the first one making use of exponential generating functions.  Recall the well-known egf formula $\frac{1}{m!}(e^x-1)^m=\sum_{n\geq m}S_{n,m}\frac{x^n}{n!}$ for a fixed non-negative integer $m$ (see, e.g., \cite[p.\,34]{Stan}).  Extracting the coefficient of $\frac{x^nu^{k+1}}{n!}$ in \eqref{Txuform} then gives
\begin{align*}
|\mathcal{R}_{n+1,k+1}|&=\left[\frac{x^nu^{k+1}}{n!}\right]\left(ue^{u(e^x-1)+(1-u)x}\right)=\left[\frac{x^nu^k}{n!}\right]\left(e^{u(e^x-1)}\cdot e^{(1-u)x}\right)\\
&=\left[\frac{x^nu^k}{n!}\right]\left(\sum_{i\geq0}\frac{u^i(e^x-1)^i}{i!}\cdot\sum_{\ell\geq0}\frac{(1-u)^\ell x^\ell}{\ell !}\right)=n!\sum_{i=0}^k[x^nu^{k-i}]\left(\frac{(e^x-1)^i}{i!}\cdot\sum_{\ell \geq0}\frac{(1-u)^\ell x^\ell}{\ell !}\right)\\
&=n!\sum_{i=0}^k\sum_{j=i}^n\frac{S_{j,i}}{j!}\cdot [u^{k-i}]\frac{(1-u)^{n-j}}{(n-j)!}=n!\sum_{i=0}^k\sum_{j=i}^n \frac{S_{j,i}}{j!(n-j)!}\cdot (-1)^{k-i}\binom{n-j}{k-i},
\end{align*}
which may be rewritten as \eqref{Rnkcarde1}.

For a combinatorial proof, first consider the set $\mathcal{W}_{i,j}=W_{i,j}^{(n,k)}$ of ordered triples $\lambda=(\alpha,\beta,\gamma)$, where $0 \leq i \leq \min\{j,k\}$ and $0 \leq j \leq n-k+i$, whose components are defined as follows:  (I) $\alpha \subseteq [n]$, with $|\alpha|=j$, (II) $\beta \subseteq [n] \backslash \alpha$,  with $|\beta|=k-i$, and (III) $\gamma$ is a partition of the elements of $\alpha$ containing exactly $i$ blocks.  Define the sign of a member of $\mathcal{W}_{i,j}$ by $(-1)^{k-i}$ and let $\mathcal{W}=\cup_{i,j}\mathcal{W}_{i,j}$.  Then the right side of \eqref{Rnkcarde1} is seen to give the sum of the signs of all the members of $\mathcal{W}$.

Given $\lambda \in \mathcal{W}$, let $p_o$ denote the largest $p \in [n]$, if it exists, such that either (i) $p \in \beta$ or (ii) $p \in \alpha$, with the singleton block $\{p\}$ occurring in $\gamma$.  If (i), then we move $p_o$ from $\beta$ to $\alpha$ and add the block $\{p_o\}$ to $\gamma$.  If (ii) occurs, then we reverse this operation deleting the block $\{p_o\}$ from $\gamma$ and moving $p_o$ from $\alpha$ to $\beta$.
Let $\lambda'$ denote the resulting member of $\mathcal{W}$ in either case.  Then it is seen that the mapping $\lambda \mapsto \lambda'$ is an involution on $\mathcal{W}$, where defined, which reverses the sign (as $i$ always changes by one).

Let $\mathcal{W}'$ denote the set of survivors $\lambda=(\alpha,\beta,\gamma)$ of this involution.  Then within $\lambda$, we must have $\beta=\varnothing$, whence $k-i=0$, with $\gamma$ containing no singleton blocks.  Thus, each member of $\mathcal{W}'$ has positive sign, with
$|\mathcal{W}'|=\sum_{j=2k}^n\binom{n}{j}T_{j,k}$, where $T_{j,k}$ denotes the associated Stirling number.  By \eqref{R_{n,k}exple1}, we then have $|\mathcal{W}'|=|\mathcal{R}_{n+1,k+1}|$, which implies \eqref{Rnkcarde1}.
\end{proof}

\subsection{Formulas for the sign-balance and total inv}

 Let $m!!=m(m-2)\cdots 1$ for an odd positive integer $m$, with $(-1)!!=1$.  We have the following formulas for the sign-balance of the inv parameter on $\mathcal{R}_n$.

 \begin{theorem}\label{sgnbal}
 If $n \geq0$, then
 \begin{align}
 A_{2n+1}(-1,1)&=\sum_{i=0}^{\lfloor n/2 \rfloor}(2i-1)!!S_{n,2i}, \label{sgnbale1}\\
A_{2n+2}(-1,1)&=\sum_{i=0}^{\lfloor n/2 \rfloor}\sum_{j=2i}^n(2i-1)!!\binom{n}{j}S_{j,2i}.\label{sgnbale2}
\end{align}
\end{theorem}
\begin{proof}
We first work more generally with $A_m(-1,u)$ for an arbitrary $u$. Letting $q=-1$ in \eqref{Anqrece1}, and observing
$$\binom{n}{k}_{-1}=\binom{\lfloor n/2\rfloor}{\lfloor k/2\rfloor}\left(1-\frac{1+(-1)^n}{2}\frac{1-(-1)^k}{2}\right),$$
we have
\begin{equation}\label{sgnbale3}
A_n(-1,u)=A_{n-1}(-1,u)+u\sum_{k=1}^{n-2}(-1)^k\binom{\lfloor (n-2)/2\rfloor}{\lfloor k/2\rfloor}\left(1-\frac{1+(-1)^n}{2}\frac{1-(-1)^k}{2}\right)A_{n-k-1}(-1,u), \qquad n \geq 2, \end{equation}
with $A_1(-1,u)=u$.  Let $E_n(u)=A_{2n}(-1,u)$ for $n \geq 1$ and $O_n(u)=A_{2n+1}(-1,u)$ for $n \geq0$.  Then \eqref{sgnbale3} implies for all $n \geq 1$:
\begin{align*}
E_n(u)&=O_{n-1}(u)+u\sum_{k=1}^{n-1}\binom{n-1}{k}O_{n-k-1}(u),\\
O_n(u)&=E_n(u)+u\sum_{k=1}^{n-1}\binom{n-1}{k}E_{n-k}(u)
-u\sum_{k=0}^{n-1}\binom{n-1}{k}O_{n-k-1}(u),
\end{align*}
with $O_0(u)=u$.  Define the generating functions $E(x;u)=\sum_{n\geq0}E_{n+1}(u)\frac{x^{n}}{n!}$ and $O(x;u)=\sum_{n\geq0}O_n(u)\frac{x^n}{n!}$. Multiplying the last two equations by $\frac{x^{n-1}}{(n-1)!}$, and summing over $n \geq 1$, we obtain
\begin{align*}
E(x;u)&=(1-u+ue^x)O(x;u),\\
\frac{\partial }{\partial x}O(x;u)&=E(x;u)+u(e^x-1)E(x;u)-ue^xO(x;u),
\end{align*}
with $O(0;u)=u$.

Substituting the expression for $E(x;u)$, and solving the resulting first-order differential equation for $O(x;u)$, yields
\begin{align*}
O(x;u)&=ue^{\frac{u^2}{2}(e^{2x}-1)+u(1-2u)(e^x-1)+(1-u)^2x}, \\
E(x;u)&=ue^{\frac{u^2}{2}(e^{2x}-1)+u(1-2u)(e^x-1)+(1-u)^2x}(1+u(e^x-1)).
\end{align*}
Taking $u=1$ gives
$$O(x;1)=e^{\frac{1}{2}(e^x-1)^2}=\sum_{i\geq0}\frac{(e^x-1)^{2i}}{2^ii!}=\sum_{i\geq0}(2i-1)!!\frac{(e^x-1)^{2i}}{(2i)!},$$
by the fact $(2i-1)!!=\frac{(2i)!}{2^ii!}$.  Hence,
$$A_{2n+1}(-1,1)=[x^n/n!]O(x;1)=\sum_{i\geq0}(2i-1)!![x^n/n!]\left(\frac{(e^x-1)^{2i}}{(2i)!}\right)=\sum_{i=0}^{\lfloor n/2 \rfloor}(2i-1)!!S_{n,2i},$$
which yields \eqref{sgnbale1}. Finally, we have $E(x;1)=e^xO(x;1)$, and thus
$$A_{2n+2}(-1,1)=[x^n/n!]E(x;1)=\sum_{j=0}^n\binom{n}{j}A_{2j+1}(-1,1),$$
which by \eqref{sgnbale1} may be rewritten as \eqref{sgnbale2}.
\end{proof}

Using the preceding formulas, one can show that the imbalance of the inv statistic becomes negligible for large $n$ when compared to the cardinality of $\mathcal{R}_n$.

\begin{corollary}\label{asymsign}
We have
$\lim_{n\rightarrow \infty}\left(\frac{A_n(-1,1)}{A_n(1,1)}\right)=0$,
and hence the inv parameter is asymptotically balanced on $\mathcal{R}_n$.
\end{corollary}
\begin{proof}
Note $n!! \leq B_n$ for $n\geq 1$ odd, with equality holding only if $n=1$, since removal of $n+1$ from its block defines a bijection between perfect matchings of $[n+1]$ and a proper subset of $\mathcal{P}_n$ if $n \geq 3$.  Thus, by \eqref{sgnbale1}, we have
$$A_{2m+1}(-1,1)\leq B_{m-1}\sum_{i=0}^{\lfloor m/2 \rfloor}S_{m,2i} \leq B_{m-1}B_m<B_{2m-1}, \qquad m \geq 2.$$
Since $A_{2m+1}(1,1)=B_{2m}$ and $\frac{B_{n}}{B_{n+1}}$ tends to zero for large $n$, the odd subsequence of ratios converges to zero.  By \eqref{sgnbale2}, we have
$$A_{2m+2}(-1,1)\leq B_{m-1}\sum_{j=0}^m\binom{m}{j}\sum_{i=0}^{\lfloor j/2 \rfloor}S_{j,2i} \leq B_{m-1}\sum_{j=0}^m \binom{m}{j}B_j=B_{m-1}B_{m+1}<B_{2m}, \qquad m \geq 2,$$
with $A_{2m+2}(1,1)=B_{2m+1}$, which implies the even subsequence converges to zero and completes the proof.
\end{proof}

There is the following asymptotic estimate for $A_{n}(-1,1)$.

\begin{theorem}\label{asymth}
As $n$ increases without bound, we have
\begin{align}
A_{2n+1}(-1,1)&\sim \sqrt{\frac{e}{2}} \left(\frac{n}{e \cdot r}\right)^n \exp\left(\frac{n}{2r} - \frac{1}{2}e^r\right),  \label{sgnbale1}\\
A_{2n+2}(-1,1)&\sim \sqrt{\frac{e}{2}}e^r \left(\frac{n}{e \cdot r}\right)^n \exp\left(\frac{n}{2r} - \frac{1}{2}e^r\right), \label{sgnbale2}
\end{align}
where $r = r_n$ is uniquely determined by $r\left(e^{2r} - e^r\right) = n$.
\end{theorem}
\begin{proof}
We may apply the saddle-point method, see \cite[Chapter VIII]{FlS}, to derive asymptotic estimates in the odd and even cases, as the corresponding generating functions $O(x;1)$ and $E(x;1)$ are seen to be Hayman-admissible, where
$$O(x;1)=e^{\frac{1}{2}(e^x-1)^2}, \quad E(x;1)=e^{x+\frac{1}{2}(e^x-1)^2}.$$
After some computation (we omit the details), one obtains the saddle-point asymptotic formula for $[x^n] O(x;1)$, which we multiply by $n!$ to get
\[
A_{2n+1}(-1,1)=n! [x^n] O(x;1) \sim \sqrt{2\pi n} \left(\frac{n}{e}\right)^n \cdot \frac{e^{1/2}}{\sqrt{4\pi n}} \cdot \frac{\exp\left(\frac{n}{2r} - \frac{1}{2}e^r\right)}{r^n},
\]
where we have made use of Stirling's approximation for $n!$.  Simplifying the last formula yields \eqref{sgnbale1} and a similar argument applies to \eqref{sgnbale2}.  Note that the saddle-point $r_n$ admits the asymptotic expansion $r_n = \frac{1}{2}\ln n - \frac{1}{2}\ln\left(\frac{1}{2}\ln n\right) + \mathcal{O}(1)$.
\end{proof}

Using the well-known Bell number estimate
$$B_n \sim n^{-1/2}(\lambda(n))^{n+\frac{1}{2}}e^{\lambda(n)-n-1},$$
where $\lambda(n)=\frac{n}{W(n)}$ and $W$ is the Lambert-$W$ function (see, e.g., \cite[p.\,17]{Lo}), one may reaffirm the asymptotic sign-balance of inv shown in Corollary \ref{asymsign} by an elementary argument and  indeed obtain an estimate of the rate at which the ratio approaches zero.

In the next result, we give a formula for the sum of the inv values taken over all members of $\mathcal{R}_n$.

\begin{theorem}\label{totinv}
If $n \geq 1$, then the total number of inversions in all the members of $\mathcal{R}_n$ is given by
\[
\frac{
B_{n+1}
-(2n+7)B_{n}
+(2n^2+4n+1)B_{n-1}
-2(n-1)(n-2)B_{n-2}
}
{8}.
\]
\end{theorem}
\begin{proof}
Define $B_n(u)=\frac{\partial}{\partial q}A_n(q,u)\mid_{q=1}$. By differentiating both sides of \eqref{Anqrece1} with respect to $q$, and setting $q=1$, we have
\begin{align}
B_n(u)&=B_{n-1}(u)
+u\sum_{k=1}^{n-2}k\binom{n-2}{k}A_{n-k-1}(1,u)+u\sum_{k=1}^{n-2}\frac{k(n-2-k)}{2}\binom{n-2}{k}A_{n-k-1}(1,u)\notag\\
&\quad+u\sum_{k=1}^{n-2}\binom{n-2}{k}B_{n-k-1}(u), \qquad n \geq 2, \label{Bnurec}
\end{align}
with $B_1(u)=0$, where we have made use of the fact $\frac{d}{dq}\binom{n}{k}_q\mid_{q=1}=\frac{k(n-k)}{2}\binom{n}{k}$.  Define $A(x;u)=\sum_{n\geq1}A_n(1,u)\frac{x^{n-1}}{(n-1)!}$ and $B(x;u)=\sum_{n\geq1}B_n(u)\frac{x^{n-1}}{(n-1)!}$. Multiplying both sides of \eqref{Bnurec} by $\frac{x^{n-2}}{(n-2)!}$, and summing over all $n \geq2$, we obtain
\begin{align*}
\frac{\partial}{\partial x}B(x;u)&=B(x;u)
+u\sum_{n\geq2}\sum_{k=1}^{n-1}A_{n-k}(1,u)\frac{x^{n-1}}{(k-1)!(n-k-1)!}\\
&\quad+\frac{u}{2}\sum_{n\geq3}\sum_{k=1}^{n-2}A_{n-k}(1,u)\frac{x^{n-1}}{(k-1)!(n-k-2)!}+u\sum_{n\geq2}\sum_{k=1}^{n-1}B_{n-k}(u)\frac{x^{n-1}}{k!(n-k-1)!}\\
&=B(x;u)+uxe^xA(x;u)+\frac{ux^2e^x}{2}\frac{\partial}{\partial x}A(x;u)+u(e^x-1)B(x;u).
\end{align*}

By \eqref{Txuform}, the last equation may be written as
\begin{equation}\label{Bxudifeq}
\frac{\partial}{\partial x}B(x;u)=(1-u+ue^x)B(x;u)+\frac{u^2xe^x}{2}\left(2+x(ue^x+1-u)\right)e^{u(e^x-1)+x(1-u)}.
\end{equation}
Solving the first-order, linear differential equation \eqref{Bxudifeq}, with the initial value $B(0;u)=0$, yields
$$B(x;u)=\frac{u^2}{8}e^{u(e^x-1)+x(1-u)}
\left[7u+e^x\left(ue^x(2x^2-2x+1)+4x^2(1-u)+8u(x-1)\right)\right],$$
and hence
\begin{equation}\label{Bx1form}
B(x;1)=\frac{1}{8}e^{e^x-1}\left(7+8(x-1)e^x+(2x^2-2x+1)e^{2x}\right).
\end{equation}
Extracting the coefficient of $\frac{x^{n-1}}{(n-1)!}$ in \eqref{Bx1form} yields the stated expression for the total number of inversions, upon making repeated use of the well-known egf formula $e^{e^x-1}=\sum_{n\geq0}B_n\frac{x^n}{n!}$ (see, e.g., \cite[p.\,34]{Stan}).
\end{proof}

\noindent \emph{Remark:} Dividing the formula in Theorem \ref{totinv} by $B_{n-1}$ yields the average number of inversions within the members of $\mathcal{R}_n$.  Using the fact $\frac{B_{n+1}}{B_n} \sim \frac{n}{\ln n}$  (see, e.g., \cite[Theorem 10]{RJ}), one has that the average number of inversions is asymptotically $\frac{n^2}{4}$.  Thus, on average, approximately half of the available pairs of indices $(i,j)$ with $1 \leq i < j \leq n$ correspond to actual inversions within a randomly chosen member of $\mathcal{R}_n$ for large $n$.

\subsection{Generating function formula}

Let $G(x;u)=G(x;q,u)$ be given by
$$G(x;u)=\sum_{n\geq0}A_{n+1}(q,u)\frac{x^n}{n_q^!}.$$
Let $e_q(x)=\sum_{n\geq0}\frac{x^n}{n_q^!}$ denote the $q$-exponential function.  By the equivalence between $\mathcal{R}_{n+1}$ and $V_n$ described above and \cite[Theorem 3.1]{QP}, we have the $q$-exponential generating function formula
\begin{equation}\label{Gxuform}
G(x;u)=\frac{u}{\prod_{j \geq0}\left(1-(1-q)q^jx(1-u+ue_q(q^{j+1}x))\right)}.
\end{equation}

In the case $u=1$, it is possible to deduce a further formula as follows.  We first recall a method of obtaining an explicit formula for a generating function based on the recurrence of the sequence.  Let $f(x)=\sum_{n\geq1}f_n\frac{x^n}{n_q^!}$ and $g(x)=\sum_{n\geq0}g_n\frac{x^n}{n_q^!}$, where $g_0=1$.  Let $h[f(x)]$ denote the $q$-composition of $q$-generating functions as defined in \cite[Definition 3.2]{Ges}.  Then we have $g(x)=e_q[f(x)]$ if and only if
\begin{equation}\label{gn+1rec}
g_{n+1}=\sum_{k=0}^n\binom{n}{k}_qf_{k+1}g_{n-k}, \qquad n \geq0,
\end{equation}
with $g_0=1$, by \cite[Proposition 3.4]{Ges}.

Let $R_n(q)=A_{n+1}(q,1)$ for $n\geq0$ and we apply the preceding to deduce a formula for
$$\mathcal{R}(x;q):=\sum_{n\geq0}R_n(q)\frac{x^n}{n_q^!}.$$
To do so, first note that letting $u=1$ in \eqref{Anqrece1} implies
$$R_{n+1}(q)=\sum_{k=0}^nq^k\binom{n}{k}_qR_{n-k}(q), \qquad n \geq0,$$
with $R_0(q)=1$.  By taking $f_k=q^{k-1}$ for $k\geq1$ in \eqref{gn+1rec}, we have
\begin{equation}\label{Rx;qform1}
\mathcal{R}(x;q)=e_q[f(x)],
\end{equation}
where $f(x)=\sum_{n\geq1}q^{n-1}\frac{x^n}{n_q^!}=\frac{e_q(qx)-1}{q}$.

Recall now the $q$-exponential formula due to Gessel (see \cite[Proposition 3.5]{Ges}) given by
\begin{equation}\label{qexpoform}
e_q[h(x)]=\prod_{j\geq0}\left(1-(1-q)q^jxh'(q^jx)\right)^{-1},
\end{equation}
where $h(0)=0$ and the prime here denotes the $q$-derivative of $h$ defined by $\frac{h(qx)-h(x)}{(q-1)x}$.  Taking $h$ to be $f$ in \eqref{qexpoform}, where $f$ is as given above, and noting
$$f'(x)=\sum_{n\geq1}q^{n-1}\frac{x^{n-1}}{(n-1)_q^!}=e_q(qx),$$
yields by \eqref{Rx;qform1} the following result.

\begin{theorem}\label{Rxqthm}
We have the $q$-exponential generating function
\begin{equation}\label{Rxqthme1}
\sum_{n\geq0}A_{n+1}(q,1)\frac{x^n}{n_q^!}=e_q\left[\frac{e_q(qx)-1}{q}\right]=\prod_{j\geq0}\left(1-(1-q)q^jxe_q(q^{j+1}x)\right)^{-1},
\end{equation}
where $e_q(x)=\sum_{n\geq0}\frac{x^n}{n_q^!}$.
\end{theorem}

Note that  \eqref{Txuform} and the first formula in \eqref{Rxqthme1} reduce to $\exp(e^x-1)$ when $u=1$ and $q=1$, respectively, which recovers the fact $|\mathcal{R}_{n+1}|=B_n$ for all $n \geq0$ in both instances. Also, we see that \eqref{Gxuform} reduces to the second formula in \eqref{Rxqthme1} when $u=1$. \medskip

\noindent \emph{Remark:} F\"{u}rlinger and Hofbauer \cite{FH} described a $q$-analogue of $B_n$ that arises as the distribution of a certain kind of inversion statistic defined on members of $\mathcal{P}_n$ represented sequentially as permutations.   Denoting this $q$-analogue by $B_n(q)$, we have in contrast with \eqref{Rxqthme1} the generating function formula
$$\sum_{n\geq0}B_n(q)\frac{x^n}{n_q^!}=e_q(e_{1/q}(x)-1);$$
see \cite[Equation~(2.10)]{FH}.

\section{Maximum number of inversions}

We have the following formula for the maximum inv value achieved on $\mathcal{R}_n$.

\begin{theorem}\label{maxinv}
Let $u_n$ denote the maximum number of inversions achieved by a member of $\mathcal{R}_n$, i.e., $u_n=\text{deg}(A_n(q,1))$.  Let $n=\binom{k}{2}+d$ for some $k \geq 2$ and $0 \leq d \leq k-1$.  Then we have
\begin{equation}\label{maxinve1}
u_n=\binom{k}{3}+3\binom{k}{4}+d\binom{k-1}{2}+\binom{d}{2},
\end{equation}
with this value being achieved by exactly $\binom{k-1}{d}+\binom{k-2}{d-2}$ members of $\mathcal{R}_n$.
\end{theorem}
\begin{proof}
We may assume $n \geq 3$, whence $k \geq 3$, since the formula is seen to give the correct value of zero in the cases $n=1$ and $n=2$.  We decompose $\rho \in \mathcal{R}_n$ into runs as $\rho=\rho^{(1)}\cdots \rho^{(\ell)}$ for some $\ell \geq 2$.  In order for $\rho$ to achieve the maximum possible number of inversions, we must have $\rho^{(1)}=1a(a+1)\cdots n$ for some $a \in [3,n]$.  For if not, and $|\rho^{(1)}|=b$ where $2 \leq b \leq n-1$, then replacing $\rho^{(1)}$ with $1(n-b+2)(n-b+3)\cdots n$, and subsequently replacing the remaining section $\rho^{(2)}\cdots\rho^{(\ell)}$ with the equivalent sequence using the letters in $[2,n-b+1]$, results in $\rho' \in \mathcal{R}_n$ with $\text{inv}(\rho')>\text{inv}(\rho)$. Continuing this reasoning in an inductive manner, it follows that $\rho^{(i+1)}$ for $1 \leq i \leq \ell-1$ must have the form $(i+1)u(u+1)\cdots v$, where $v=\min\{\rho^{(1)}\cdots\rho^{(i)} \backslash [i]\}-1$ and $\pi^{(\ell)}$ is possibly of length one.  Thus, we may restrict our search to $\rho$ having the stated form, the subset of $\mathcal{R}_n$ of which we denote by $\mathcal{R}_n'$.

Let $\pi=\pi^{(1)}\cdots \pi^{(r)} \in \mathcal{R}_n'$ for some $r \geq 1$, where $\pi^{(i)}$ denotes the $i$-th run of $\pi$, and we seek to maximize $\text{inv}(\pi)$.  Note that $\pi$ is uniquely determined by its sequence of run cardinalities $a_i=|\pi^{(i)}|$ for $1 \leq i \leq r$.  We will subsequently identify $\pi \in \mathcal{R}_n'$ with the associated composition of $n$ given by $a_\pi=(a_1,\ldots,a_r)$.  Note, for all $\pi$, that each of the parts of $a_\pi$ is at least two, except for possibly the last. Given $a_\pi$, define the sequence $b_\pi$ by
\[
 b_\pi =
\begin{cases}
    (a_1-k+1,a_2-k+2,\ldots,a_{k-1}-1,a_k,a_{k+1},\ldots,a_r),   & \text{if } r \geq k; \\
    (a_1-k+1,a_2-k+2,\ldots,a_r-k+r),    & \text{if } r \leq k-1.
\end{cases}
\]

Let $\mathcal{R}$ denote the subset of $\mathcal{R}_n'$ consisting of those $\pi$ for which $r=k-1$ or $k$ such that $b_\pi$ contains only 0's and 1's.  Let $\mathcal{T}=\mathcal{R}_n'\backslash \mathcal{R}$, and we show that no member of $\mathcal{T}$ can achieve the maximum inv. Let $\pi \in \mathcal{T}$, with $b_\pi=(b_1,\ldots,b_r)$, where we first assume $r \geq k-1$. Let $M=\max(b_i)_{1 \leq i \leq r}$, $m=\min(b_i)_{1 \leq i \leq r}$ and $w=M-m$.  Then $r \geq k-1$ implies $w \geq 2$.  To show this, first note $\pi \in \mathcal{T}$ implies $M \geq 1$ and we consider cases on $M$.  If $M \geq 2$, then there must exist an index $i \in [k-1]$ such that $b_i \leq 0$, whence $w \geq 2$, for if not, then
$$\sum_{i=1}^r a_i=\binom{k}{2}+\sum_{i=1}^r b_i\geq \binom{k}{2}+k>n,$$
which is impossible. If $r \geq k+1$, then the component $b_k$ is at least two, as $b_i=a_i$ for $i \geq k$.  Thus, if $M=1$, then we must have $r=k-1$ or $k$ and $\pi \in \mathcal{T}$ implies $m<0$, whence $w \geq 2$.

Let $s, t \in [r]$ such that $b_s=m$ and $b_t=M$, where we assume $s$ and $t$ are the respective smallest such indices.  Suppose now $\pi \in \mathcal{T}$ with $r \geq k-1$ and $M \geq 2$, where if $M=2$, then $t \leq k-1$.  We show that the maximum value of inv on $\mathcal{R}_n'$ is not achieved by such $\pi$. To do so, let $\pi^*$ denote the member of $\mathcal{R}_n'$ which arises from $\pi$ by changing $a_\pi$ as follows:  replace $a_s$ with $a_s+1$ and $a_t$ with $a_t-1$, keeping all of the other components of $a_\pi$ the same.  Note $a_t \geq 3$, by the assumptions on $\pi$, whence $\pi^*$ also has $r$ runs.

We now show $\text{inv}(\pi^*)>\text{inv}(\pi)$, and thus $\pi$ is not maximal.  To do so, we consider cases on $s$ and $t$. If $s>t$, then changing the $s$-th and $t$-th components of $a_\pi$ as described above has the effect of losing $\sum_{i=t+1}^s a_i$ inversions, while creating $a_t-2+\sum_{i=t+1}^{s-1}(a_i-1)$. Thus, since $w\geq 2$ and $t<s \leq k-1$, we obtain a net gain of
\begin{align*}
 a_t-2+\sum_{i=t+1}^{s-1}(a_i-1)-\sum_{i=t+1}^s a_i&=a_t-a_s-(s-t+1)=(a_t-(k-t))-(a_s-(k-s))-1=M-m-1\geq 1
 \end{align*}
 inversions in the transition from $\pi$ to $\pi^*$, whence $\text{inv}(\pi^*)>\text{inv}(\pi)$, as desired.   Now suppose $s<t \leq k-1$.  In this case, we obtain a net gain of
 $$a_t-1+\sum_{i=s+1}^{t-1}a_i-\sum_{i=s}^{t-1}(a_i-1)=a_t-a_s+t-s-1=w-1 \geq 1$$
 inversions, leading to the same conclusion.  If $t \geq k$, then we must have $M \geq 3$, by the assumptions on $\pi$.  One may then perform the same replacements as before for a net gain of
 $$a_t-a_s+t-s-1\geq M-(a_s-(k-s))-1=w-1 \geq 1$$
 inversions, which completes the cases on $s$ and $t$.

 Now suppose $\pi \in \mathcal{T}$, where $t \geq k$ and $M=2$.  In this case, we replace $a_s$ with $a_s+1$ and either replace $a_t=2,\,a_{t+1}$ with $a_{t+1}+1$ if $t<r$ or $a_t=2$ with $a_t=1$ if $t=r$.  Note that we are decreasing the number of runs by one in the first case, where inversions in addition to those considered already would be created by elements in the new run corresponding to $a_{t+1}+1$ if $t<r-1$.  Changing $a_\pi$ as described then results in a net gain of at least
 $$a_t-a_s+t-s-1\geq 1-(a_s-(k-s))=1-m\geq 1$$
 inversions, since $M=2$ and $w \geq 2$ implies $m \leq 0$.

 Next, assume $ \pi \in \mathcal{T}$, where  $r \geq k-1$ and $M=1$.  Note $M=1$ implies $r=k-1$ or $k$ in this case, since $b_k=a_k\geq 2$ if $r \geq k+1$. We consider cases on $t$.  If $t \leq k-2$ or if $t=r=k-1$, then since $w \geq 2$ we can proceed as above in the case when $M \geq 2$ and both $s$ and $t$ belonged to $[k-1]$.  If $r=k$ and $t=k-1$ or $k$, then replace $a_s$ with $a_s+1$ and delete $a_r=1$.  One may verify in all cases that the resulting member of $\mathcal{R}_n'$ has strictly more inversions than $\pi$.

 Now assume $\pi \in \mathcal{T}$, with $r \leq k-2$.  Note $M\geq 1$, for otherwise, the components of $a_\pi$ would add up to at most $n-1$.  Let $p \in [r]$ such that $b_p=a_p-(k-p)\geq 1$.  Let $\pi^*$ be obtained from $\pi$ in this case by replacing $a_p$ with $a_p-1$ and adding the component $a_{r+1}=1$ to $a_\pi$, assuming for now $a_r>1$.  This results in a net increase of
 $$a_p-2+\sum_{i=p+1}^r(a_i-1)-\sum_{i=p+1}^ra_i=a_p-(r-p+2)\geq a_p-(k-p)\geq 1$$
 inversions, whence $\text{inv}(\pi^*)>\text{inv}(\pi)$.  If $a_r=1$, then we replace $a_p$ with $a_p-1$ and change $a_r$ from 1 to 2.  By the same reasoning, one has that this operation results in a member of $\mathcal{R}_n'$ with strictly more inversions.

 Combining all of the cases above shows that the maximum inv value on $\mathcal{R}_n'$, and hence on $\mathcal{R}_n$, must be achieved on $\mathcal{R}$.  We now compute inv for members of $\mathcal{R}$.  First suppose $n=\binom{k}{2}$, where $k \geq 3$. In this case, it is seen that $\mathcal{R}$ consists of a single member $\sigma \in \mathcal{R}_n'$ such that $a_\sigma=(k-1,k-2,\ldots,1)$.  Note that if $i \in [k-2]$, then each of the $k-i$ numbers, except for the first, within the $i$-th run of $\sigma$ corresponds to the larger entry in a total of $\sum_{j=1}^{k-i-1}j=\binom{k-i}{2}$ inversions.  Considering all possible $i$ then implies
 $$\text{inv}(\sigma)=\sum_{i=1}^{k-2}(k-i-1)\binom{k-i}{2}=\sum_{i=1}^{k-2}i\binom{i+1}{2}=\binom{k}{3}+3\binom{k}{4},$$
 which establishes the $d=0$ case of \eqref{maxinve1}.

 So assume $n=\binom{k}{2}+d$ for some $d \in [k-1]$.  Given a statement $P$, let $[P]=1$ or $0$, depending on the truth or falsehood of $P$.  In this case, the set $\mathcal{R}$ consists of $\tau$ for which $a_\tau$ is expressible as either
 $$a_\tau=(k-1+[1 \in S],k-2+[2 \in S],\ldots,1+[k-1 \in S]),$$
 for some $S \subseteq [k-1]$ of size $d$, or
 $$a_\tau=(k-1+[1 \in T],k-2+[2\in T],\ldots,2+[k-2 \in T],2,1),$$
 for some $T \subseteq [k-2]$ of size $d-2$ (provided $d \geq 2$).  We now compute $\text{inv}(\tau)$ for $\tau$ of the first form above.  To do so, we add to $\text{inv}(\sigma)$, where $\sigma$ is as above, the number of extra inversions arising when the run lengths for runs of $\sigma$ corresponding to members of $S$ are increased by one. Suppose $1 \leq i_1<\cdots<i_d \leq k-1$ represent the members of $S$.  Then run $i_1$ of $\tau$ is of length $k-i_1+1$, with the extra added element in this run involved in $\binom{k-1}{2}+d-1$ inversions not accounted for by $\text{inv}(\sigma)$.  If $2 \leq j \leq d$, then the extra element in the run $i_j$ of $\tau$ is seen to contribute $\binom{k-1}{2}+d-j$ inversions that involve this element and that have not been previously accounted for by runs prior to $i_j$ or by $\text{inv}(\sigma)$.

 This implies
 $$\text{inv}(\tau)=\text{inv}(\sigma)+\sum_{j=1}^d\left(\binom{k-1}{2}+d-j\right)=\binom{k}{3}+3\binom{k}{4}+d\binom{k-1}{2}+\binom{d}{2}.$$
 Similar considerations show that $\text{inv}(\tau)$ is given by this same formula if $a_\tau$ is of the second form above.  Thus, for all $n$, each member of $\mathcal{R}$ achieves the same (maximum) value of inv as stated in \eqref{maxinve1}, with $|\mathcal{R}|=\binom{k-1}{d}+\binom{k-2}{d-2}$,  which completes the proof.
 \end{proof}

 \noindent \emph{Remark:}  We noticed the article by Beveridge et al. \cite{Bev}, which appeared in the archives soon after we had finished an initial draft of the current paper, where it was found independently by the authors a formula for the maximum number of inversions in 32-1 avoiding permutations of length $n$, equivalently, in run-sorted permutations of length $n+1$.  Their expression was shown to coincide with A023536 in \cite{Sl} and was obtained after an in-depth study of a geometric representation of the Lehmer code for a 32-1 avoiding permutation, which they called a \emph{jump-float sequence}. Here, we have provided an alternative proof by making use of runs that is shorter than the one given in \cite{Bev} and yields a different formula for the maximum number of inversions.  Further, our argument may be extended so as to obtain an expression for the maximum inv value on $\mathcal{R}_{n,k}$ for each $k$ (see Theorem \ref{maxinvrnk} below) and to establish the unimodality of this sequence over $k$ for a fixed $n$ (see Corollary \ref{unimod}). \medskip

Given $j \geq 1$ and $n \geq 2j-1$, let $m_{n,j}$ denote the maximum number of inversions achieved by a member of $\mathcal{R}_{n,j}$.  Upon using the appropriate parts from the proof of Theorem \ref{maxinv}, one can establish the unimodality of $m_{n,j}$.

\begin{corollary}\label{unimod}
Let $n=\binom{k}{2}+d$, where $k \geq 3$ and $0 \leq d \leq k-1$.  Then we have $m_{n,j}<m_{n,j+1}$ if $j \leq k-2$ and $m_{n,j-1}>m_{n,j}$ if $j \geq k+1$.  In particular, the sequence $m_{n,j}$ for $1 \leq j \leq \lfloor (n+1)/2 \rfloor$ and $n$ fixed is unimodal.
\end{corollary}
\begin{proof}
Let $\mathcal{R}_{n,j}'=\mathcal{R}_n' \cap \mathcal{R}_{n,j}$, where $\mathcal{R}_n'$ (and other notation used here) is as before.  Let $\pi \in \mathcal{R}_{n,j}'$, with $a_\pi=(a_1,\ldots,a_j)$.  First suppose $j \in [k-2]$ and we show $m_{n,j}<m_{n,j+1}$.   To do so, it suffices to exhibit $\pi^* \in \mathcal{R}_{n,j+1}'$ with $\text{inv}(\pi^*)>\text{inv}(\pi)$.  The case in the proof of Theorem \ref{maxinv} wherein $\pi \in \mathcal{T}$ and $r \leq k-2$ provides such a $\pi^*$ if $a_j \geq 2$.  So assume $a_j=1$ and note $M\geq 1$ implies we must have  $b_t \geq 1$ for some $t \leq j-1$.  In this case, we first replace $a_t$ by $a_t-1$ and change $a_j$ from 1 to 2 to obtain $\widehat{\pi}$.  Using $\widehat{\pi}$ whose last run is of length two, we proceed as in the prior case to obtain $\pi^* \in \mathcal{R}_{n,j+1}'$.  One may verify $\text{inv}(\pi^*)>\text{inv}(\widehat{\pi})>\text{inv}(\pi)$.

Now suppose $\pi \in \mathcal{R}_{n,j}$ where $j \geq k+1$ and we seek $\pi^* \in \mathcal{R}_{n,j-1}'$ such that $\text{inv}(\pi^*)>\text{inv}(\pi)$.  Then $j \geq k+1$ implies there must exist $s \in [k-1]$ such that $b_s \leq 0$ within $b_\pi=(b_1,\ldots,b_j)$.  Replacing $a_s$ by $a_s+1$ and  $a_j$ by $a_j-1$ within $a_\pi$ is seen to result in a member of $\mathcal{R}_{n}'$ with strictly greater inv.  We then repeat this operation until all of the letters occurring in the final run of $\pi$ have been exhausted, where the part $a_s$ that is increased may need to change from one step to the next.  Let $\pi^*$ denote the member of $\mathcal{R}_{n,j-1}'$ that arises from the resulting composition with $j-1$ parts.  One may verify $\text{inv}(\pi^*) \geq \text{inv}(\pi)+\binom{\ell+1}{2}$, where $\ell=a_j$.  This implies $m_{n,j-1}>m_{n,j}$ for $j \geq k+1$ and completes the proof of the first statement.  The unimodality in $j$ of the sequence $m_{n,j}$ for a fixed $n$ then follows from the first statement.  Note further that the proof of Theorem \ref{maxinv} shows that the maximum value of $m_{n,j}$ always occurs when $j=k-1$, with $j=k$ maximal as well if and only if $ d\geq 2$ in the representation of $n$.
\end{proof}

The sequence $u_n$ may also be expressed without cases on $n$ as above in terms of the following summation.

\begin{theorem}\label{maxinvpr}
If $n \geq 1$, then
\begin{equation}\label{maxinvpre1}
u_{n+2}=\frac{n^2+5n-4}{2}-\sum_{j=2}^n\lfloor 1/2+\sqrt{2j+4}\rfloor.
\end{equation}
\end{theorem}
\begin{proof}
We may assume $n \geq 4$, as one may verify \eqref{maxinvpre1} directly for $1 \leq n \leq 3$.  Let $v_n$ denote the expression on the right-hand side of \eqref{maxinvpre1}.  We first show the equality $v_n=u_{n+2}$ when $n=\binom{k}{2}-2$, using the formula for $u_n$ in \eqref{maxinve1}.  To compute $v_n$, where $n=\binom{k}{2}-2$ for some $k \geq 4$, we first rewrite the sum strategically as
\begin{equation}\label{maxinvpre2}
\sum_{j=2}^{\binom{k}{2}-2}\lfloor 1/2+\sqrt{2j+4}\rfloor=\sum_{i=3}^{k-1}\sum_{j=\binom{i}{2}-1}^{\binom{i+1}{2}-2}\lfloor 1/2+\sqrt{2j+4}\rfloor.
\end{equation}
Now observe that if $i \geq 3$ and $\binom{i}{2}-1 \leq j \leq \binom{i+1}{2}-2$, then
$$\left(i-\frac{1}{2}\right)^2 < 2j+4 < \left(i+\frac{1}{2}\right)^2,$$
and hence $\lfloor 1/2+\sqrt{2j+4}\rfloor=i$ for all such $j$. Substituting this into \eqref{maxinvpre2} implies
$$\sum_{j=2}^{\binom{k}{2}-2}\lfloor 1/2+\sqrt{2j+4}\rfloor=\sum_{i=3}^{k-1}\sum_{j=\binom{i}{2}-1}^{\binom{i+1}{2}-2} i=\sum_{i=3}^{k-1}i^2=\frac{k(k-1)(2k-1)}{6}-5.$$
Therefore, we have
\begin{align*}
v_{\binom{k}{2}-2}&=\frac{(\binom{k}{2}-2)^2+5(\binom{k}{2}-2)-4}{2}-\left(\frac{k(k-1)(2k-1)}{6}-5\right)=\frac{\binom{k}{2}^2+\binom{k}{2}}{2}-\frac{k(k-1)(2k-1)}{6}\\
&=\binom{k}{3}+3\binom{k}{4}=u_{\binom{k}{2}},
\end{align*}
which establishes $v_n=u_{n+2}$ in the case $n=\binom{k}{2}-2$, as desired.

To complete the proof of the equality $v_n=u_{n+2}$, it thus suffices to show
\begin{equation}\label{maxinvpre3}
v_n-v_{n-1}=u_{n+2}-u_{n+1}, \qquad \binom{k}{2}-1 \leq n \leq \binom{k}{2}+k-3,
\end{equation}
where $k \geq 4$.  Represent $n$ within the range stated in \eqref{maxinvpre3} by  $n=\binom{k}{2}+b-2$, where $1 \leq b \leq k-1$.
By the prior observations, we have $\lfloor 1/2+\sqrt{2n+4}\rfloor =k$ for all such $n$ and hence
\begin{align*}
v_n-v_{n-1}&=\frac{n^2+5n-4}{2}-\frac{(n-1)^2+5(n-1)-4}{2}-\lfloor 1/2+\sqrt{2n+4}\rfloor=n+2-k=\binom{k-1}{2}+b-1.
\end{align*}
On the other hand, by \eqref{maxinve1}, we have
\begin{align*}
u_{n+2}-u_{n+1}&=\binom{k}{3}+3\binom{k}{4}+b\binom{k-1}{2}+\binom{b}{2}-\left(\binom{k}{3}+3\binom{k}{4}+(b-1)\binom{k-1}{2}+\binom{b-1}{2}\right)\\
&=\binom{k-1}{2}+b-1
\end{align*}
for such $n$. This implies \eqref{maxinvpre3} and completes the proof of \eqref{maxinvpre1}.
\end{proof}

\noindent \emph{Remark:} Theorem \ref{maxinvpr} shows $u_{n}=A023536(n-2)$ for all $n \geq 3$.  It is known that $A023536(n-2)$ gives the number of possible values for the number of diagonals in a convex polyhedron with $n+1$ vertices and it would be interesting to provide a direct proof that this equals the maximum number of inversions within a member of $\mathcal{R}_{n}$. \medskip

We have the following explicit formulas for $m_{n,k}$.

\begin{theorem}\label{maxinvrnk}
If $2k-1 \leq n <\binom{k+1}{2}$, then
\begin{equation}\label{maxinvrnke1}
m_{n,k}=k^2+(r^2-r+2d-2)k+3\binom{r+1}{4}-2\binom{r+1}{3}+(d-1)\binom{r}{2}-dr+\frac{d^2-5d+2}{2},
\end{equation}
where $n=2k-1+\binom{r}{2}+d$ with $1 \leq r \leq k-2$ and $0 \leq d <r$. If $n \geq \binom{k+1}{2}$, then
\begin{align}
m_{n,k}&=\frac{3k+3r+2}{4}\binom{k+r+1}{3}-\frac{3r-2}{4}\binom{r+1}{3}-\binom{k+r-d+1}{3}-\left(rk+\binom{k}{2}+d\right)\binom{r+1}{2}\notag\\
&\quad+(k+r)\binom{d}{2}-\binom{d}{3}, \label{maxinvrnke2}
\end{align}
where $n=\binom{k+1}{2}+rk+d$ with $r \geq 0$ and $0 \leq d <k$.  Furthermore, the value of $m_{n,k}$ is achieved by $\binom{r}{d}$ members of $\mathcal{R}_{n,k}$ in the first case and $\binom{k}{d}$ members of $\mathcal{R}_{n,k}$ in the second.
\end{theorem}
\begin{proof}
We modify the proof of Theorem \ref{maxinv}.  Since the procedures applied to non-maximal members of $\mathcal{R}_n$ described in that proof did not always preserve the number of runs, special care must be taken to obtain the maximal members of $\mathcal{R}_{n,k}$.  Let $\mathcal{R}_{n,k}'=\mathcal{R}_n'\cap \mathcal{R}_{n,k}$, where $\mathcal{R}_{n}'$ is as in the proof of Theorem \ref{maxinv}, and we may restrict attention to members of $\mathcal{R}_{n,k}'$.  As before, we represent $\pi \in \mathcal{R}_{n,k}'$ by a composition $a_\pi=(a_1,\ldots,a_k)$ wherein the $i$-th part of $a_\pi$ gives the length of the $i$-th run of $\pi$.  We first show \eqref{maxinvrnke1}, noting that $2k-1 \leq n <\binom{k+1}{2}$ for a fixed $k \geq 3$ may be uniquely represented by $k$, $r$ and $d$ in this case as indicated where $r$ and $d$ satisfy the stated restrictions.  If $r=1$, then $d=0$ and $n=2k-1$, with  $\mathcal{R}_{2k-1,k}'$ seen to consist of a single permutation $\pi$ whose associated composition is $(2,\ldots,2,1)$.  Note $\text{inv}(\pi)=\sum_{i=1}^{k-1}(2i-1)=(k-1)^2$, which agrees with the $r=1,\,d=0$ case of \eqref{maxinvrnke1}.

So let us assume $k \geq 4$, with $2 \leq r \leq k-2$, for the remainder of the proof of \eqref{maxinvrnke1}. Given $\pi \in \mathcal{R}_{n,k}'$ with $a_\pi=(a_1,\ldots,a_k)$, let $u_\pi=(u_1,\ldots,u_k)$ wherein $u_i=a_i-2$ for $1 \leq i < k$, with $u_k=a_k-1$, and let $v_\pi=(v_1,\ldots,v_k)$ wherein $v_i=u_i-(r-i)$ for $1 \leq i \leq r-1$, with $v_i=u_i$ for $r \leq i \leq k$.  Let $\mathcal{R}^{(k)}$ denote the subset of $\mathcal{R}_{n,k}'$ consisting of those members such that $v_i=0$ or $1$ for $1 \leq i \leq r$, with $v_{r+1}=\cdots=v_k=0$, and let $\mathcal{T}^{(k)}=\mathcal{R}_{n,k}'\backslash \mathcal{R}^{(k)}$.  We show that no member of $\mathcal{T}^{(k)}$ can achieve the maximum inv value of $m_{n,k}$. Let $\pi \in \mathcal{T}^{(k)}$ and first suppose $u_j \geq 1$ for some $j \in [r+1,k]$ within $u_\pi$.  Note we must have $v_i \leq 0$ for some $i \in [r-1]$, for otherwise
$\sum_{i=1}^ka_i \geq 2k-1+\binom{r}{2}+r>n$, which is impossible.  We then replace $a_j$ with $a_j-1$ and $a_i$ with $a_i+1$ within $a_\pi$ and consider the resulting member $\pi^* \in \mathcal{R}_{n,k}'$.

Proceeding as before, we have $\text{inv}(\pi^*)=\text{inv}(\pi)+a_j-a_i+j-i-1$.  Note $v_i \leq 0$ implies $a_i \leq r-i+2$, and hence
$$a_j-a_i+j-i-1 \geq a_j-(r-i+2)+j-i-1=a_j-3+j-r \geq 1,$$
as $u_j \geq 1$ implies $a_j \geq 3$ if $r<j<k$, with $a_k \geq 2$ if $j=k \geq r+2$.  Thus, $\text{inv}(\pi^*)>\text{inv}(\pi)$ and $\pi$ is not maximal, as desired.  So assume $u_j=0$ for each $j \in [r+1,k]$ within $u_\pi$.  Let $M=\max(v_i)_{1 \leq i \leq r}$ and $m=\min(v_i)_{1 \leq i \leq r}$.   Upon considering separately the cases (i) $M \geq 2$ and (ii) $M=1$ with $m<0$, one can show $\text{inv}(\pi)<m_{n,k}$ for the remaining $\pi \in \mathcal{T}^{(k)}$, by applying an argument comparable to the one used at a similar juncture in the proof of Theorem \ref{maxinv}.  Note $M \geq 2$ in (i) implies $m \leq 0$, for otherwise the sum of the components of $a_\pi$ would exceed $n$.

One may verify that all members of $\mathcal{R}^{(k)}$ have equal inv value, with $|\mathcal{R}^{(k)}|=\binom{r}{d}$.  We first compute $\text{inv}(\sigma)$ when $d=0$, where $\sigma$ denotes the sole member of $\mathcal{R}^{(k)}$ in this case.  Note $\sigma \in \mathcal{R}^{(k)}$ together with  $n=2k-1+\binom{r}{2}$  implies $u_\sigma=(r-1,r-2\ldots,1,0,\ldots,0)$.  Upon considering separately the inversions involving the second elements in each of the runs of $\sigma$, we have
\begin{align*}
\text{inv}(\sigma)&=\sum_{i=1}^{k-1}(2i-1)+\sum_{i=0}^{r-3}(r-i)\binom{r-i-1}{2}+\sum_{i=1}^{r-1}(r-i)(2(k-i)-1)\\
&=(k-1)^2+\sum_{i=1}^{r-1}(i+1)\binom{i}{2}+\sum_{i=1}^{r-1}i(2(k-r+i)-1)\\
&=(k-1)^2+\frac{3}{2}\sum_{i=1}^{r-1}i^2+\frac{1}{2}\sum_{i=1}^{r-1}i^3 +(2k-2r-1)\binom{r}{2}+\binom{r}{3}\\
&=(k-1)^2+\frac{r(r-1)(2r-1)}{4}+\frac{1}{2}\binom{r}{2}^2+(2k-2r-1)\binom{r}{2}+\binom{r}{3}.
\end{align*}
This last expression may be rewritten somewhat to give
\begin{equation}\label{d=0case}
\text{inv}(\sigma)=(k-1)^2+kr(r-1)+3\binom{r+1}{4}-2\binom{r+1}{3}-\binom{r}{2},
\end{equation}
which yields the $d=0$ case of \eqref{maxinvrnke1}.

Now suppose $\tau \in \mathcal{R}^{(k)}$, where $d \in [r-1]$.  We may assume $a_\tau$ is obtained from $a_\sigma$ be adding 1 to each of the first $d$ entries of $a_\sigma$. Considering separately from the others those inversions that arise due to this addition of 1 to the first $d$ entries of $a_\sigma$, we have
\begin{align*}
\text{inv}(\tau)&=\text{inv}(\sigma)+\binom{d}{2}+\sum_{i=1}^d\binom{r-i}{2}+\sum_{i=1}^d(2k-2i-1)+\sum_{i=0}^{d-2}(r-i)(d-i-1)\\
&=\text{inv}(\sigma)+\binom{d}{2}+\binom{r}{3}-\binom{r-d}{3}+d(2k-1)-d(d+1)+\sum_{i=1}^{d-1}i(r-d+i+1)\\
&=\text{inv}(\sigma)+\binom{d}{2}+\binom{r}{3}-\binom{r-d}{3}+d(2k-1)-d(d+1)+(r-d)\binom{d}{2}+2\binom{d+1}{3}\\
&=\text{inv}(\sigma)+2dk+\binom{r}{3}-\binom{r-d}{3}+r\binom{d}{2}-\frac{d(d^2+17)}{6}.
\end{align*}
Using formula \eqref{d=0case} for $\text{inv}(\sigma)$, and combining some terms, we get
$$\text{inv}(\tau)=k^2+(r^2-r+2d-2)k+3\binom{r+1}{4}-\binom{r+1}{3}-2\binom{r}{2}-\binom{r-d}{3}+r\binom{d}{2}-\frac{d^3+17d-6}{6}.$$
Applying the fact
$$r\binom{d}{2}-\binom{r-d}{3}=d\binom{r}{2}+\binom{d+2}{3}-\binom{r}{3}-dr$$
to the last expression yields \eqref{maxinvrnke1} after some algebraic steps.

We now show \eqref{maxinvrnke2}. First note $n \geq \binom{k+1}{2}$ may be uniquely represented in terms of $k$, $r$ and $d$ as described.  Further, formula \eqref{maxinvrnke2} is seen to give the correct value of zero for all $n \geq 1$ when $k=1$, so we may assume $k \geq 2$.  Let $\pi \in \mathcal{R}_{n,k}'$, where $n=\binom{k+1}{2}+rk+d$ and $r$ and $d$ are as stated.  If $a_\pi=(a_1,\ldots,a_k)$, then define the vector $w_\pi=(w_1,\ldots,w_k)$ wherein $w_i=a_i-k-r+i-1$ for $1 \leq i \leq k$. Let $\mathcal{R}^{(k)}$ in this case denote the subset of $\mathcal{R}_{n,k}'$ consisting of those $\pi$ such that $w_\pi$ consists of only $0$'s and $1$'s and let $\mathcal{T}^{(k)}=\mathcal{R}_{n,k}'\backslash \mathcal{R}^{(k)}$.  Note $d<k$ in the representation of $n$ implies $w_\pi$ cannot be all $1$'s. Let $M=\max(w_i)_{1 \leq i \leq k}$ and $m=\min(w_i)_{1 \leq i \leq k}$, and note $M \geq 1$ with $M -m \geq 2$ for all members of $\mathcal{T}^{(k)}$.  Proceeding as before considering $M$ and $m$, one can show that no member of $\mathcal{T}^{(k)}$ can achieve the maximum inv value on $\mathcal{R}_{n,k}$.  Further, it is seen that all members of $\mathcal{R}^{(k)}$ achieve the same (maximum) inv value, with $|\mathcal{R}^{(k)}|=\binom{k}{d}$.

Let $\sigma$ denote the sole member of $\mathcal{R}^{(k)}$, where $n$ is of the stated form with $d=0$, and let $\tau$ be a member of $\mathcal{R}^{(k)}$, where $d \in [k-1]$ in the representation of $n$.  Note $a_\sigma=(k+r,k+r-1,\ldots,r+1)$ and we may assume $\tau$ arises by adding $1$ to each of the first $d$ entries of $a_\sigma$ and considering the resulting permutation.  We first compute $\text{inv}(\sigma)$.  Upon considering the number of inversions caused by the letters in the $(k-j-1)$-st run of $\sigma$ for $0 \leq j \leq k-2$, we have
\begin{align*}
\text{inv}(\sigma)&=\sum_{j=0}^{k-2}(j+r+1)\sum_{i=r+1}^{j+r+1}i=\sum_{j=0}^{k-2}(j+r+1)\left(\binom{j+r+2}{2}-\binom{r+1}{2}\right)\\
&=\sum_{j=0}^{k-2}(j+r+1)\binom{j+r+2}{2}-\binom{r+1}{2}\sum_{j=0}^{k-2}(j+r+1)\\
&=3\sum_{j=0}^{k-2}\binom{j+r+3}{3}-2\sum_{j=0}^{k-2}\binom{j+r+2}{2}-\left(r(k-1)+\binom{k}{2}\right)\binom{r+1}{2}\\
&=3\binom{k+r+2}{4}-3\binom{r+3}{4}-2\binom{k+r+1}{3}+2\binom{r+2}{3}-\left(r(k-1)+\binom{k}{2}\right)\binom{r+1}{2}\\
&=\binom{k+r+2}{4}+2\binom{k+r+1}{4}-\binom{r+3}{4}-2\binom{r+2}{4}-\left(r(k-1)+\binom{k}{2}\right)\binom{r+1}{2}.
\end{align*}
Using the fact $\binom{m}{4}+2\binom{m-1}{4}=\frac{3m-8}{4}\binom{m-1}{3}$ for $m \geq 3$ in the last expression implies
\begin{align}
\text{inv}(\sigma)&=\frac{3k+3r-2}{4}\binom{k+r+1}{3}-\frac{3r+1}{4}\binom{r+2}{3}-\left(r(k-1)+\binom{k}{2}\right)\binom{r+1}{2}\notag\\
&=\frac{3k+3r-2}{4}\binom{k+r+1}{3}-\frac{3r-2}{4}\binom{r+1}{3}-\left(rk+\binom{k}{2}\right)\binom{r+1}{2},\label{d=0case(b)}
\end{align}
which establishes the $d=0$ case of \eqref{maxinvrnke2}.

Considering the inversions caused by the extra elements added to the first $d$ runs of $\sigma$ in obtaining $\tau$, we have
$$\text{inv}(\tau)=\text{inv}(\sigma)+\binom{d}{2}+\sum_{j=1}^d\sum_{i=r+1}^{k+r-j}i+\sum_{i=0}^{d-2}(d-i-1)(k+r-i-1).$$
Note
$$\sum_{j=1}^d\sum_{i=r+1}^{k+r-j}i=\sum_{j=1}^d\left(\binom{k+r-j+1}{2}-\binom{r+1}{2}\right)=\binom{k+r+1}{3}-\binom{k+r-d+1}{3}-d\binom{r+1}{2}$$
and
$$\sum_{i=0}^{d-2}(d-i-1)(k+r-i-1)=(k+r-d)\binom{d}{2}+\frac{d(d-1)(2d-1)}{6}=(k+r)\binom{d}{2}-\binom{d+1}{3},$$
which implies
$$\text{inv}(\tau)=\text{inv}(\sigma)+\binom{k+r+1}{3}-\binom{k+r-d+1}{3}-d\binom{r+1}{2}+(k+r)\binom{d}{2}-\binom{d}{3}.$$
Formula \eqref{maxinvrnke2} now follows from \eqref{d=0case(b)}, which completes the proof.
\end{proof}

The proof of Theorem \ref{maxinv} shows that if $n=\binom{k}{2}+d$ for some $k \geq 2$ and $0 \leq d \leq k-1$, then the maximum number of inversions on $\mathcal{R}_{n}$ is achieved by some $\pi \in \mathcal{R}_{n,k-1}$ or $\mathcal{R}_{n,k}$.  In the first case, if $0 \leq d \leq k-2$, then there is agreement between formula \eqref{maxinve1} for $u_n$ and the special case of \eqref{maxinvrnke2} wherein $r=0$ and $k$ is replaced by $k-1$, as expected.  Similarly, if $d=k-1$, then the $r=1,\,d=0$ case of \eqref{maxinvrnke2} with $k-1$ in place of $k$ coincides with \eqref{maxinve1}. On the other hand, if the maximum inv on $\mathcal{R}_n$ is achieved by some $\pi \in\mathcal{R}_{n,k}$, then the proof of Theorem \ref{maxinv} shows $d \geq 2$ in this case.  Then the special case of formula \eqref{maxinvrnke1} with $r=k-2$ and $d$ replaced by $d-2$ is seen to coincide with \eqref{maxinve1}, as expected.  Further, the number of members of $\mathcal{R}_n$ for which the maximum inv is achieved is seen to arise from the corresponding numbers of such members of $\mathcal{R}_{n,k-1}$ or $\mathcal{R}_{n,k}$ in these cases.

Note that the $m_{n,k}$ values for the various $k$ are difficult to compare, as  the representations of $n$ used in Theorem \ref{maxinvrnk} strongly depend upon $k$ and change, sometimes drastically, in going from $k$ to $k+1$.  As a result, it seems unlikely that one can obtain Theorem \ref{maxinv} simply by identifying the $k$ for which  $m_{n,k}$ is greatest for a fixed $n$ using the expressions in Theorem \ref{maxinvrnk}.  Further, the argument given above in Corollary \ref{unimod} is needed to deduce the unimodality of the $m_{n,k}$ due to the complexity of these expressions.  Finally, we note that the proof of Theorem \ref{maxinv} shows that the maximum number $u_n$ of inversions on $\mathcal{R}_n$ is always achieved by a member of $\mathcal{R}_{n,m}$, where $m$ is the largest $k$ such that $\binom{k+1}{2} \leq n$.  Moreover, if $n$ is a triangular number or one greater than such a number, then this is the only $k$ for which $u_n$ is achieved on  $\mathcal{R}_{n,k}$.  For all other $n$, the maximum $u_n$ is achieved on $\mathcal{R}_{n,m+1}$ as well but nowhere else.


\footnotesize\begin{thebibliography}{20}

\bibitem{ARV}
M. Ahmia, J. L. Ram\'{\i}rez and D. Villamizar, Inversions in colored permutations, derangements and involutions, \emph{Adv. in Appl. Math.} {\bf 173} (2026), \#102999.

\bibitem{ANab}
P. Alexandersson and O. Nabawanda, Peaks are preserved under run-sorting, \emph{Enum. Combin. Appl.} {\bf 2}:1 (2022), Art. \#S2R2.

\bibitem{BHR}
J.-L. Baril, P. E. Harris and J. L. Ram\'{\i}rez, Flattened Catalan words, \emph{Bull. Inst. Combin. Appl.} {\bf 104} (2025), 51--79.

\bibitem{Bev}
A. Beveridge, K. Heysse and P. Robertson, \underline{32}1-avoiding permutations with maximum inversion number, arXiv:2607.18417v1, pre-print, 2026.

\bibitem{BEF}
A. Buck, J. Elder, A. A. Figueroa, P. E. Harris, K. J. Harry and A. Simpson, Flattened Stirling permutations, \emph{Integers} {\bf 24} (2024), \#A104.

\bibitem{Callan}
D. Callan, Pattern avoidance in ``flattened'' partitions, \emph{Discrete Math.} \textbf{309} (2009), 4187--4191.

\bibitem{EHM}
J. Elder, P. E. Harris, Z. Markman, I. Tahir and A. Verga, On flattened parking functions, \emph{J. Integer Seq.} {\bf 26} (2023), Art. 23.5.8.

\bibitem{FlS}
P. Flajolet and R. Sedgewick, \emph{Analytic Combinatorics}, Cambridge University Press, Cambridge, UK, 2009.

\bibitem{FH}
J. F\"{u}rlinger and J. Hofbauer, $q$-Catalan numbers, \emph{J. Combin. Theory Ser. A} {\bf 40} (1985), 248--264.

\bibitem{Ges}
I. M. Gessel, A $q$-analog of the exponential formula, \emph{Discrete Math.} {\bf 40} (1982), 69--80.

\bibitem{GR}
J. Goldman and G.-C. Rota, The number of subspaces of a vector space, \emph{Recent Progress in Combinatorics}, Academic Press, New York (1969), 75--83.

\bibitem{RJ}
R. Jakimczuk, Integer sequences, functions of slow increase, and the Bell numbers, \emph{J. Integer Seq.} {\bf 14} (2011), Art. 11.5.8.

\bibitem{Lo}
L. Lov\'{a}sz, \emph{Combinatorial Problems and Exercises}, Second Edition, AMS Chelsea Publishing, Providence, RI, 2007.

\bibitem{SMa1}
S.-M. Ma, An explicit formula for the number of permutations with a given number of alternating runs, \emph{J. Combin. Theory Ser. A} {\bf 119}:8 (2012), 1660--1664.

\bibitem{SMa2}
S.-M. Ma, Enumeration of permutations by number of alternating runs, \emph{Discrete Math.} {\bf 313}:18 (2013), 1816--1822.

\bibitem{Mac}
P. A. MacMahon, \emph{Combinatory Analysis}, Cambridge University Press, Cambridge, UK, 1915-16.

\bibitem{MSh}
T. Mansour and M. Shattuck, Counting subword patterns in permutations arising as flattened partitions of sets,
  \emph{Appl. Anal. Discrete Math.} {\bf 16}:1 (2022), 146--177.

\bibitem{Mar}
B. H. Margolius, Permutations with inversions, \emph{J. Integer Seq.} {\bf 4} (2001), Art. 01.2.4.

\bibitem{Muir}
T. Muir, On a simple term of a determinant, \emph{Proc. Roy. Soc. Edinburgh} {\bf 21} (1898-9), 441--477.

\bibitem{NR}
O. Nabawanda and F. Rakotondrajao, Sets of flattened partitions with forbidden patterns, \emph{J. Integer Seq.} {\bf 24} (2021), Art. 21.1.5.

\bibitem{NRB}
O. Nabawanda, F. Rakotondrajao and A. S. Bamunoba, Run distribution over flattened partitons, \emph{J. Integer Seq.} {\bf 23} (2020), Art. 20.9.6.

\bibitem{QP}
Q. Pan, Inversion-descent enumerators of \underline{32}1-avoiding permutations, arXiv:2608.30266v1, pre-print, 2026.


\bibitem{RUC}
V. R. Rao Uppuluri and J. A. Carpenter, Numbers generated by the function $\exp(1-e^x)$, \emph{Fibonacci Quart.} {\bf 7}(4) (1969), 437--448.

\bibitem{Sl}
N. J. A. Sloane et al., \emph{The On-Line Encyclopedia of Integer Sequences}, published electronically at http://oeis.org, 2026.

\bibitem{Stan}
R. P. Stanley, {\em Enumerative Combinatorics, Vol. I}, Cambridge University Press, Cambridge, UK, 1997.

\bibitem{Wag}
C. G. Wagner, \emph{A First Course in Enumerative Combinatorics}, American Mathematical Society, Providence, RI, 2020.

\bibitem{Ya}
Y. Yang, On a multiplicative partition function, \emph{Electron. J. Combin.} {\bf 8}:1 (2001), \#R19.

\bibitem{Zh}
Y. Zhuang, Counting permutations by runs, \emph{J. Combin. Theory Ser. A} {\bf 142} (2016), 147--176.


\end{thebibliography}
\end{document}